\documentclass[a4paper,12pt]{article}
\usepackage[dvipdfmx]{graphicx}
\usepackage{ascmac}
\usepackage{amsmath}
\usepackage{amsthm}
\usepackage{amsfonts}
\usepackage[papersize={210truemm, 297truemm}, top=1in,bottom=1in,left=1in,right=1in]{geometry}
\usepackage{latexsym}

\newcommand{\bi}[1]{\mbox{\boldmath $#1$}}
\newcommand{\haf}{\operatorname{haf}}
\newcommand{\pf}{\operatorname{pf}}
\newtheorem{theo}{Theorem}[section]

\newtheorem{prop}[theo]{Proposition}
\newtheorem{lem}[theo]{Lemma}

\title{Shortest $(A+B)$-path packing via hafnian}
\author{Hiroshi HIRAI and Hiroyuki NAMBA\\
	Department of Mathematical Informatics,\\
	 Graduate School of Information Science and Technology,\\
	 The University of Tokyo, Tokyo, 113-8656, Japan.\\
	\texttt{\normalsize hirai@mist.i.u-tokyo.ac.jp} \\
	\texttt{\normalsize hiroyukinannba@gmail.com}}
\begin{document}
\maketitle
\begin{abstract}
Bj\"{o}rklund and Husfeldt developed a randomized polynomial time algorithm to solve the shortest two disjoint paths problem. Their algorithm is based on computation of permanents modulo 4 and the isolation lemma. In this paper, we consider the following generalization of the shortest two disjoint paths problem, and develop a similar algebraic algorithm.
The shortest perfect $(A+B)$-path packing problem is:
given an undirected graph $G$ and two disjoint node subsets $A,B$ with even cardinalities,
find  shortest $|A|/2+|B|/2$ disjoint paths whose ends are both in $A$ or both in $B$.
Besides its NP-hardness, we prove that this problem can be solved in randomized polynomial time if $|A|+|B|$ is fixed.
Our algorithm basically follows the framework of Bj\"{o}rklund and Husfeldt but uses a new technique: computation of hafnian modulo $2^k$ combined with
Gallai's reduction from $T$-paths to matchings.
We also generalize our technique for solving other path packing
problems, and discuss its limitation.

\vspace{6mm}
\noindent \textbf{Keywords:}\hspace{2mm} shortest disjoint paths problem, hafnian, randomized polynomial time algorithm
\end{abstract}

\section{Introduction}
{\em The shortest two disjoint paths problem} is: given an undirected graph $G = (V,E)$ and $s_1,t_1,s_2,t_2 \in V$, find two disjoint paths, one connecting $s_1$ and $t_1$ and the other connecting~$s_2$ and $t_2$, such that the sum of their lengths is minimum.
Although the length-less version, the two disjoint paths problem, is elegantly solved \cite{seymour,shiloach,thomassen},
no polynomial time algorithm was known for this generalization.
Recently, Bj\"{o}rklund and Husfeldt \cite{twopaths} obtained the first polynomial time algorithm.

\begin{theo}[\cite{twopaths}]\label{bh}
There exists a randomized polynomial time algorithm to solve the shortest two disjoint paths problem. 
\end{theo}

Their algorithm is build on striking application of computation of permanents modulo~$4$ by Valiant \cite{valiant} and the isolation lemma by Mulmuley--Vazirani--Vazirani~\cite{vazirani}. 

In this paper, we consider a generalization of the shortest two disjoint paths problem
and develop a randomized polynomial time algorithm based on a similar algebraic technique. 
Let us introduce our problem. 
For $T\subseteq V$, a $T$-$path$ is a path connecting distinct nodes in $T$. 
We are given two disjoint terminal sets $A$ and $B$ with even cardinalities. 
A {\em perfect (A+B)-path packing} is a set ${\cal P}$ of node-disjoint paths such that each path is an $A$-path or $B$-path and $|{\cal P}|=|A|/2+|B|/2$. 
The {\em size} of a perfect $(A+B)$-path packing is defined as the total sum of the length of each path, where the length of a path is defined as the number of edges in the path. 
{\em The shortest perfect $(A+B)$-path packing problem} asks to find a perfect $(A+B)$-path packing with minimum size. 
It will turn out that this problem is NP-hard. 
In the case where $|A|=|B|=2$, the problem is the shortest two disjoint paths problem above. 
When $B$ is empty, the problem is the disjoint $A$-path problem by Gallai \cite{gallai}. 
Our main result says that the problem is tractable, provided $|A|+|B|$ is fixed. 
\begin{theo}\label{theo:main}
There exists a randomized algorithm to solve the shortest perfect $(A+B)$-path packing problem in $O(f(|V|)^{|A|+|B|})$ time, where $f$ is a polynomial. 
\end{theo}
Our algorithm basically follows the framework of Bj\"{o}rklund--Husfeldt \cite{twopaths} but we use a new technique: computation of hafnian modulo $2^k$, instead of permanent modulo $4$, combined with a classical reduction technique  to matching by Gallai (for $T$-paths) \cite{gallai} and Edmonds (for odd path); see \cite[Section 29.11e]{schrijver}.

\paragraph{Related work}
Colin de Verdi\`ere--Schrijver \cite{schrijver2} and Kobayashi--Sommer \cite{kobayashi2} gave combinatorial polynomial time algorithms for shortest disjoint paths problems in planar graphs with special terminal configurations. 
Karzanov \cite{karzanov} and Hirai--Pap \cite{hirai} showed the polynomial time solvability of a shortest version of edge-disjoint $T$-paths problem. 
Yamaguchi~\cite{yamaguchi} reduced the shortest disjoint ${\cal S}$-paths problem (nonzero $T$-paths problem in a group labeled graph, more generally) to weighted matroid matching. 
Kobayashi--Toyooka~\cite{kobayashi} developed a randomized polynomial time algorithm for the shortest nonzero $(s,t)$-path problem in a group labeled graph; their algorithm is also based on the framework of Bj\"{o}rklund--Husfeldt. 

It is well-known that the hafnian of the adjacency matrix of a graph is equal to the number of perfect matchings. 
By utilizing the hafnian,
Bj\"{o}rklund~\cite{bjorklundSODA12} 
developed a faster algorithm 
to count the number of perfect matchings. 

\paragraph{Organization}
The rest of this paper is organized as follows. 
In Section 2, we first show that hafnian modulo $2^k$ for fixed $k$ is computable in polynomial time.
This direct generalization of permanent computation modulo $2^k$ seems new and interesting in its own right.
Next we present the randomized algorithm in Theorem \ref{theo:main}. 
In Section 3, we verify the hardness of the $(A+B)$-path packing problem, and then generalize our technique for solving other path packing problems, and discuss its limitation.

\section{Algorithm}\label{sec:haf}
In this section, we first provide an algorithm to compute hafnian modulo $2^k$, and next present a randomized polynomial time algorithm to solve the shortest perfect $(A+B)$-path packing problem for fixed $|A|+|B|$. 
An undirected pair or edge $\{i,j\}$ is simply denoted by~$ij$.

\subsection{Computing Hafnian Modulo $2^k$}
The {\em hafnian} $\haf A$ of a $2n\times 2n$ symmetric matrix $A=(a_{ij})$ is defined by
\begin{equation*}\haf A:=\sum_{M\in {\cal M}}\prod_{ij\in M}a_{ij}, 
\end{equation*}
where ${\cal M}$ is the set of all partitions of $\{1,2,3,\dots,2n\}$ into $n$ pairs.

Let ${\cal S}(n,N)$ denote the set of all $2n\times 2n$ symmetric matrices with zero diagonal each of whose element is a univariate polynomial of degree at most $N$. 
Let $\haf_{2^k} A$ denote the hafnian of $A$ modulo $2^k$. 
The main result of this subsection is the following: 

\begin{theo}\label{theo:hafmod2k}
There exists a bivariate polynomial $f$ such that
for all $A\in {\cal S}(n,N)$, $\haf_{2^k} A$ can be computed in  $O(f(n,N)^k)$ time.
\end{theo}
We prove Theorem $\ref{theo:hafmod2k}$ by the similar way to
that for permanents modulo $2^k$ \cite{valiant} and that for permanents of polynomial matrices modulo $2^k$~\cite{twopaths,kobayashi}. 
First we verify Theorem $\ref{theo:hafmod2k}$ for $k=1$. 
Let $\tilde A=(\tilde{a}_{ij})$ be a skew-symmetric matrix obtained from $A$ by replacing~$a_{ij}$ by $-a_{ij}$ if $i > j$. 
Modulo $2$, $\haf A$ coincides with $\pf \tilde{A}$ (Pfaffian of $\tilde{A}$). 
Hence $\haf_2 A$ can be obtained in time polynomial in $n$ and $N$
by computing $\sqrt{\det \tilde{A}}$ (mod $2$).

Next, we consider the case of $k\geq 2$. 
We use a formula like the Laplace expansion of determinants. Let $A[i,j]$ denote the matrix obtained from $A$ by removing the row $i$, row $j$, column $i$, and column $j$. For distinct $i,j,p,q$, let $A[i,j,p,q]:=(A[i,j])[p,q]$. 
\begin{lem}\label{lem:laplace}
\begin{itemize}
\item[\rm (1)]
$\displaystyle\haf A=\sum_{j:\,j\neq i}a_{ij}\haf A[i,j]$. \\
\item[\rm (2)] 
$\haf A=a_{ij}\haf A[i,j]+\displaystyle\sum_{pq:\,p,q\not\in\{i,j\}, p\neq q }(a_{ip}a_{jq}+a_{iq}a_{jp})\haf A[i,j,p,q]$. 
\end{itemize}
\end{lem}
\begin{proof}
(1)\:For $j\neq i$, let ${\cal M}_j$ be the set of all $M\in{\cal M}$ that contain $ij$. Since $\{{\cal M}_j\mid j\neq i\}$ is a partition of ${\cal M}$, we obtain
\begin{equation*}\haf A:=\sum_{j:\,j\neq i}a_{ij}\sum_{M\in{\cal M}_j}\prod_{pq\in M\backslash\{ij\}}a_{pq}=\sum_{j:\,j\neq i}a_{ij}\haf A[i,j]. 
\end{equation*}
(2)\:By using (1) repeatedly, we obtain
\begin{eqnarray*}
\haf A&=&\sum_{p:\,p\neq i}a_{ip}\haf A[i,p]
=a_{ij}\haf A[i,j]+\sum_{p:\,p\not\in\{i,j\}}a_{ip}\haf A[i,p]\\
&=&a_{ij}\haf A[i,j]+\sum_{p:\,p\not\in\{i,j\}}a_{ip}\sum_{q:\,q\not\in\{i,j,p\}}a_{jq}\haf A[i,j,p,q]\\
&=&a_{ij}\haf A[i,j]+\displaystyle\sum_{(p,q):\,p,q\not\in\{i,j\},p\neq q}a_{ip}a_{jq}\haf A[i,j,p,q]. 
\end{eqnarray*}
Combining the terms for $(p,q)$ and $(q,p)$, we obtain (2). 
\end{proof}

For $A\in {\cal S}(n,N)$, let $A(i,j;c)$ denote the matrix obtained from $A$ by adding $c$ multiple of column $i$ to column $j$, adding $c$ multiple of row $i$ to row $j$, and replacing the $jj$th element with zero. 
We refer to this operation as the {\em $(i,j;c)$-operation}. 
Note that differences between $A$ and $A(i,j;c)$ occur only in row $j$ and column $j$, and that $A(i,j;c)$ also belongs to $S(n,N)$.
We investigate how the hafnian changes by the $(i,j;c)$-operation. 
Let $A(i\to j)$ denote the matrix obtained from $A$ by replacing row $j$ with row $i$ and column $j$ with column $i$.

\begin{lem}\label{lem:hafhenka}
$\haf A(i,j;c)=\haf A+c\:\haf A(i\to j)$. 
\end{lem}
\begin{proof}
Let $\tilde{a}_{pq}$ denote the $pq$th element of $A(i,j;c)$. We use Lemma \ref{lem:laplace} (1) with respect to row~$j$ and column $j$.
\begin{eqnarray*}
	&&\haf A(i,j;c)= \sum_{k: k \neq j} \tilde a_{kj} \haf A[k,j] \\[3pt]
  && =\sum_{k:\,k\neq j}a_{kj}  \haf A[k,j] +\sum_{k:\,k\neq j} c a_{ki} \haf A[k,j] \\[3pt]
	&&= \haf A+c\:\haf A(i\to j). 
\end{eqnarray*}
%\begin{eqnarray*}
%&&\haf A(i,j;c)=\sum_{M}\prod_{pq\in M}\tilde{a}_{pq}
% =\sum_{k:\,k\neq j}\tilde{a}_{kj}\sum_{M:\,kj\in M}\prod_{pq\in M\backslash \{kj\}}\tilde{a}_{pq}\\[3pt]
% &&=\sum_{k:\,k\neq j}a_{kj}\sum_{M:\,kj\in M}\prod_{pq\in M\backslash \{kj\}}a_{pq}+\sum_{k:\,k\neq j}ca_{ki}\sum_{M:\,kj\in M}\prod_{pq\in M\backslash \{kj\}}a_{pq}\\[3pt]
%&&= \haf A+c\:\haf A(i\to j). 
%\end{eqnarray*}
\end{proof}
Let $d$ be a fixed positive integer. 
A term of a polynomial is said to be {\em lower} if its degree is at most $d$ and {\em higher} otherwise.
A polynomial $f$ is said to be {\em even}
if all coefficients of lower terms of the polynomial $f(x)$ are even.
For a polynomial $f(x)$ that is not even, let $m(f(x))$ denote the lowest degree of terms with odd coefficients.

Let $A=(a_{ij}) \in {\cal S}(n,d)$. We are going to show that all lower terms of $\haf A$ modulo $2^k$ can be computed in time polynomial in $n$ and $d$. 
The hafnian does not change if we exchange row $i$ and row $j$, and column $i$ and column $j$.
Hence we exchange rows and columns of $A$ in advance so that $a_{12}$ is a minimizer of $m(a_{1j})$ in $a_{1j}\:(j=2,\dots, 2n)$ that are not even. 
Next we find a polynomial $c_j$ such that $c_ja_{12}+a_{1j}$ is even for $j=3,\dots 2n$. The computation can easily be done in time polynomial in $n$ and $d$ \cite[Section 3.2]{twopaths}. 
Using the $(2,j;c_j)$-operation for $j=3,\dots 2n$ in order, we obtain matrices $A_3:=A(2,3;c_3),\:A_4:=A_3(2,4;c_4),\dots, A_{2n}:=A_{2n-1}(2,2n;c_{2n})$. 
Then $1j$ elements of $A_{2n}$ are even if $j\geq 3$. 
Applying Lemma $\ref{lem:hafhenka}$ repeatedly, we obtain
\begin{equation*}\haf A_{2n}=\haf A+\sum_{j=3}^{2n}c_j\haf A_{j-1}(2\to j), 
\end{equation*}
where $A_2=A$. Using Lemma \ref{lem:laplace} (1) for $A_{2n}=(b_{ij})$, we obtain
\begin{equation}\haf A=b_{12}\haf A_{2n}[1,2]+\sum_{j=3}^{2n}b_{1j}\haf A_{2n}[1,j]-\sum_{j=3}^{2n}c_j\haf A_{j-1}(2\to j). 
\label{zenka}\end{equation}
Though there may be higher terms in elements of matrices in (\ref{zenka}), we may replace these higher terms with $0$ (since our goal is computing lower terms). 
Similarly we may replace higher terms in $b_{1j}\:(j=2,\dots,2n)$ with $0$.
Hence all matrices in right-hand side of (\ref{zenka}) can be regarded in ${\cal S}(n-1,d)$ or ${\cal S}(n,d)$. 

Next we discuss the second and third terms of the right-hand side in detail.  
For the second term, we obtain $b_{1j}\haf A_{2n}[1,j]$ modulo $2^k$ from $\haf A_{2n}[1,j]$ modulo $2^{k-1}$ since $b_{1j}\:(3\leq j\leq 2n)$ are even. 
Therefore we need to compute hafnians of $2n-2$ polynomial matrices in ${\cal S}(n-1,d)$ modulo $2^{k-1}$. 

Next we consider the third term. 
For $A(i\to j)$, it holds $a_{ip}=a_{jp},\:a_{iq}=a_{jq}$ and $a_{ij}=0$ 
(since $A$ has zero diagonals).
Hence, applying Lemma \ref{lem:laplace} (2) to $A(i\to j)$, we obtain the following: 
\begin{equation*}
\haf A(i\to j)=\sum_{p,q}2a_{ip}a_{jq}\haf A[i,j,p,q]. 
\end{equation*}
Hence we obtain $\haf A(i\to j)$ modulo $2^k$ from
 hafnians of $\binom{2n-2}{2}$ matrices in ${\cal S}(n-2,d)$
modulo $2^{k-1}$. 

In this way, our algorithm recursively computes lower terms of $\haf A$ modulo $2^k$ according to (\ref{zenka}).
We are now ready to prove Theorem \ref{theo:hafmod2k}. 
\begin{proof}[Proof of Theorem $\ref{theo:hafmod2k}$]
Let $T(n,d,k)$ be the computational complexity of computing all lower terms of the hafnian of a matrix in ${\cal S}(n,d)$. 
From (\ref{zenka}) and the argument after (\ref{zenka}), it follows
\begin{eqnarray*}T(n,d,k)&\leq&T(n-1,d,k)+(2n-2)T(n-1,d,k-1)\\&&+(2n-2)\binom{2n-2}{2}T(n-2,d,k-1)+\mathrm{poly}(n,d), \end{eqnarray*}
where $\mathrm{poly}(n,d)$ is a polynomial of $n$ and $d$. 
Since $T(n,d,k)$ is monotone increasing on~$n$, it follows that
\begin{equation*} T(n,d,k)\leq T(n-1,d,k)+4n^3T(n,d,k-1)+\mathrm{poly}(n,d). \end{equation*}
Using this inequality repeatedly, we obtain
\begin{equation*} 
T(n,d,k)\leq 4n^4T(n,d,k-1)+\mathrm{poly}(n,d). 
\end{equation*}
$T(n,d,1)$ is a polynomial of $n$ and $d$ by the result of the case $k=1$. 
Hence there exists a polynomial $f$ of $n$ and $d$ such that for all positive integers $k$, $T(n,d,k)$ is $O(f(n,d)^{k})$. 
%Replacing $d$ with $nN$, we obtain the desired result. 

For $A\in {\cal S}(n,N)$, the degree of $\haf A$ is at most $nN$. 
Apply the above algorithm with $d=nN$, we obtain $\haf_{2^k} A$ in $O(f(n,nN)^{k})$ time. This completes the proof. 
\end{proof}

\subsection{Perfect $(A+B)$-Path Packing via Hafnian}\label{sec:algo}
Let $G =(V,E)$ be a simple undirected graph and $A,B$ disjoint node sets of even cardinalities. Let $n:=|V|$ and  $m:=|E|$. 
We can assume that $G=(V,E)$ has no edge with both endpoints in $A \cup B$; otherwise, replace each edge by a series of two edges. 
We consider a general case where $G$ has positive integer weight $w(e)$ on each edge $e$. 
We assume that the maximum value of the weight is bounded by a polynomial of $n$. 
For a path $P$, let $w(P)$ denote the sum of the weight of edges in $P$. 
The {\em size} of a set ${\cal P}$ of vertex-disjoint paths is defined as the total sum of $w(P)$ over $P\in{\cal P}$, and is  denoted by $w({\cal P})$. 

\paragraph{Gallai's construction}
From input $G,A,B$, we construct graph $H=(V_H,E_H)$ so that matchings in $H$ correspond to disjoint $T$-paths in $G$ (with $T = A \cup B$).
This construction is due to Gallai \cite{gallai}; see \cite[Section 73.1]{schrijver}. 
Let $U:= V\backslash (A\cup B)$. 
First we add to $G$ a copy of the subgraph of $G$ induced by $U$. 
The copy of a node $v\in U$ is denoted by $v'$.
Let $U':=\{v'\mid v\in U\}$, $V_H:=V \cup U'=A\cup B\cup U\cup U'$. 
Next, for each $v\in U$, add an edge $vv'$. 
The set of such edges is denoted by $E_=$. 
Finally, we add edge $uv'$ for each $uv\in E$ with $u\in A\cup B ,v\in U$. The set of all edges in $A\cup B\cup U'$ is denoted by $E'$. Let $E_H:=E\cup E'\cup E_=$. 
The weight $w$ is extended to $E_{H}\to\mathbb{Z}_{\geq 0}$ by
$$\begin{cases}
w(e):= 0 & \mathrm{if}\:e\in E_=,\\
w(uv'):=w(uv)& \mathrm{if}\:uv'\in E', u\in A\cup B, \\
w(u'v'):=w(uv)& \mathrm{if}\:u'v'\in E', u',v'\in U'. 
\end{cases}$$
A {\em perfect $(A\cup B)$-path packing} is a set of $|A|/2+|B|/2$ node-disjoint $(A\cup B)$-paths.
From a perfect matching $M$ of $H$, we obtain a perfect $(A\cup B)$-path packing ${\cal P}_M$ in $G$ as follows. For all $s\in A\cup B$, there exists a unique path $P=\{s,v_1,v_2,\dots,t\}$ in $H$ such that $(s,v_1)\in M$, $t\in (A\cup B)\backslash\{s\}$ and it goes through edges in $M$ and edges in $E_=$ alternately.
This path in $H$ determines an $(s,t)$-path in $G$ by picking up the only nodes in $(A\cup B)\cup U$ in the same order. 
Gathering up these paths, we obtain a perfect $(A\cup B)$-path packing ${\cal P}_M$ in $G$. 
Conversely, one can see that any perfect $(A\cup B)$-path packing in $G$ is obtained in this way. 
The size of ${\cal P}_M$ is at most the weight of $M$. 
They coincide if and only if all edges of $M$ not used by ${\cal P}_M$
belong to $E_=$. 

\paragraph{Matrices $S$ and $S'$}
Next we introduce a symmetric matrix $S$ associated with $H$.
Let $h:=|V_H|$. 
We can assume that $V_H=\{1,2,\dots, h\}$. 
Let $S=(s_{ij})$ be an $h\times h$ symmetric matrix defined by 
\begin{equation*}
s_{ij}:=
\begin{cases}x^{w(ij)}&\mathrm{if}\:ij\in E_{H}, \\0&\mathrm{otherwise}. \end{cases}
\end{equation*}
Recall that $w(ij)$ denotes the weight of the edge $ij$ in $H$. 

For $t\in A\cup B$, let $E_{t}$ denote the set of edges joining $t$ and $U$,
and let $E'_{t}$ denote the set of edges joining $t$ and $U'$. 
From the matrix $S$, we define a new matrix $S'=(s'_{ij})$ by
\begin{equation*}
s'_{ij}:=
\begin{cases}-s_{ij}&\mathrm{if}\:ij\in E'_{t}\:\mathrm{for\:some}\:t\in B, \\s_{ij}&\mathrm{otherwise}. \end{cases}
\end{equation*}
Let $\tau:=(|A|+|B|)/2$. For a perfect $(A+B)$-path packing 
${\cal P}$, let $\theta({\cal P})$ denote the number of even-length $B$-paths in ${\cal P}$. 
\begin{lem}\label{lem:hafs}
\begin{equation*}
\haf S'=\displaystyle\sum_{{\cal P}}(-1)^{\theta({\cal P})}2^{\tau}x^{w({\cal P})}(1+xf_{\cal P}(x)),
\end{equation*}
where ${\cal P}$ ranges over all perfect $(A+B)$-path packings, and $f_{\cal P}(x)$ is a polynomial.  
\end{lem}
\begin{proof}
For a matching $M$ of $H$, let $s'(M):=\prod_{ij\in M}s'_{ij}$. By the above discussion on Gallai's construction, we obtain
\begin{equation}\label{equ1}
\haf S'=\sum_M s'(M)
 =\displaystyle\sum_{{\cal P}}\sum_{M: {\cal P}_M={\cal P}}s'(M), 
\end{equation}
where $M$ ranges over all perfect matchings in $H$ and ${\cal P}$ ranges over all perfect $(A\cup B)$-path packings in $G$. 
First we estimate $\sum_{M:{\cal P}_M={\cal P}}s'(M)$. Suppose ${\cal P}=\{P_1,\dots,P_{\tau}\}$. 
For each path $P_{k}=(s_k,v_1,v_2,\dots,v_{n_k},t_k)\:(k=1,\dots,\tau)$, we define two matchings $M_{k,1},M_{k,2}$ in $H$ by
\begin{eqnarray*}
M_{k,1}=
\begin{cases}
\{s_kv_1,v'_1v'_2,\dots,v_{n_k-1}v_{n_k},v'_{n_k}t_k\}& \mathrm{if}\:n_k\:\mathrm{is}\:\mathrm{odd}, \\
\{s_kv_1,v'_1v'_2,\dots,v'_{n_k-1}v'_{n_k},v_{n_k}t_k\}& \mathrm{if}\:n_k\:\mathrm{is}\:\mathrm{even},  
\end{cases}\\
M_{k,2}=
\begin{cases}
\{s_kv'_1,v_1v_2,\dots,v'_{n_k-1}v'_{n_k},v_{n_k}t_k\}& \mathrm{if}\:n_k\:\mathrm{is}\:\mathrm{odd}, \\
\{s_kv'_1,v_1v_2,\dots,v_{n_k-1}v_{n_k},v'_{n_k}t_k\}& \mathrm{if}\:n_k\:\mathrm{is}\:\mathrm{even}. 
\end{cases}
\end{eqnarray*}
Both of them have weight $w(P_{k})$. 
Then a perfect matching $M$ with ${\cal P}_M={\cal P}$ can be represented as the union of $\bigcup_{k=1}^{\tau}M_{k,i_k}\:(i_k\in\{1,2\})$ and a perfect matching $M'$ of the subgraph $H - {\cal P}$ of $H$ obtained 
by removing vertices in $\bigcup_{k=1}^{\tau}M_{k,i_k}$. 
Then we obtain
\begin{eqnarray}
\nonumber\sum_{M:{\cal P}_M={\cal P}}s'(M)&=&\sum_{i_1\in\{1,2\}}\cdots\sum_{i_{\tau}\in\{1,2\}}\sum_{M'}s'(M_{1,i_1})\cdots s'(M_{\tau,i_{\tau}})s'(M')\\\label{equ2}
&=&(s'(M_{1,1})+s'(M_{1,2}))\cdots(s'(M_{\tau,1})+s'(M_{\tau,2}))\sum_{M'}s'(M'),
\end{eqnarray}
where $M'$ ranges over all perfect matchings of $H - {\cal P}$.

Next we estimate $s'(M_{k,1})+s'(M_{k,2})$. 
We call an edge in $E'_{t}$ for $t\in B$ {\em minus}. 
Then $s'(M_{k,j})=x^{w(P_k)}$ if $M_{k,j}$ has an even number of minus edges, and
$s'(M_{k,j})=-x^{w(P_k)}$ if $M_{k,j}$ has an odd number of minus edges. 
If $P_k$ connects $A$ and $B$, just one of $M_{k,1}$ and $M_{k,2}$ contains one minus edge. 
If $P_k$ is an $A$-path, then neither $M_{k,1}$ nor $M_{k,2}$ contains one minus edge. 
If $P_k$ is a $B$-path and the length of $P_k$ is odd, one of $M_{k,1}$ and $M_{k,2}$ has two minus edges and the other has no minus edge. 
If $P_k$ is a $B$-path and the length of $P_k$ is even, both of $M_{k,1}$ and $M_{k,2}$ have one minus edge. 
(Recall the assumption that there is no edge joining $A\cup B$.)
Hence we obtain
\begin{equation}\label{equ3}
s'(M_{k,1})+s'(M_{k,2})=
\begin{cases}
0 & \mbox{if $P_k$ connects $A$ and $B$},\\
-2x^{w(P_k)}&\mbox{if $P_k$ is an even-length $B$-path},\\
2x^{w(P_k)}&\mbox{otherwise}.
\end{cases}
\end{equation}
Finally we estimate $\sum_{M'} s'(M')$. The perfect matching consisting of edges in $E_=$ has weight $0$, and other perfect matchings have weight at least $1$. 
Thus $\sum_{M'}s'(M')$ is represented as $1+xf(x)$ for a polynomial $f$. By this fact and equations (\ref{equ1}), (\ref{equ2}) and (\ref{equ3}), we obtain the formula.
\end{proof} 

\paragraph{Unique Optimal Solution Case.}
We first consider the case where $G$ has a unique shortest perfect $(A+B)$-path packing ${\cal P}^*$. Here $w$ is not necessarily uniform (but is bounded by a polynomial of $n$).
In this case, Lemma \ref{lem:hafs} immediately yields a desired algorithm to find ${\cal P}^*$.
Indeed, the leading term (lowest degree term) of $\haf S'$ is $(-1)^{\theta({\cal P}^*)} 2^{\tau}x^{w({\cal P}^*)}$ (by the uniqueness).
In particular we can obtain the minimum degree $w({\cal P}^*)$ by computing $\haf S'$ modulo $2^{\tau +1}$.
Observe that an edge $e$ belongs to ${\cal P}^*$ if and only if
the degree of the leading term of $\haf S'$ strictly increases when $e$ is removed from $G$.
Thus we can determine ${\cal P}^*$ by $m+1$ computations of the hafnian of a $2n \times 2n$ matrix in modulo $2^{\tau+1}$. By Theorem \ref{theo:hafmod2k} (with $N=$ maximum of $w$),
this can be done in $O(f(n)^{|A|+|B|})$ time for a polynomial~$f$.

\paragraph{General Case.}
Suppose now that $w$ is uniform weight, i.e., $w(e)=1$ for all $e$ in $E$.
We consider the general case where there may be two or more shortest perfect $(A+B)$-path packings. 
We construct a randomized polynomial time algorithm with the help of the isolation lemma~\cite{vazirani}. 
This technique is due to \cite{twopaths}. We use the isolation lemma in the following form: 

\begin{lem}\label{lem:isolemma}
Let $n$ be a positive integer and ${\cal F}$ a family of subsets of $E=\{e_1,\dots,e_m\}$. 
Weight $w(e_i)$ is assigned to each element $e_i$ of $E$, where  $w(e_i)$ are chosen  independently and uniformly at random from $\{2mn,2mn+1,\dots,2mn+2m-1\}$. 
Then, with probability greater than $1/2$, there exists a unique set $F\in {\cal F}$ of minimum weight $w(F):=\sum_{e\in F}w(e)$. 
\end{lem}
We are ready to prove our main theorem. 
\begin{proof}[Proof of Theorem $\ref{theo:main}$]
We perturb the weight $w$ into $w'$ so that a shortest packing for $w'$ is unique and is also shortest for $w$.
For each edge $e$, choose $a$ from $\{2mn,\dots,2mn+2m-1\}$ independently and uniformly at random, and let $w'(e):= a$. 
By Lemma~\ref{lem:isolemma}, with a high probability ($\geq 1/2$), a shortest $(A+B)$-path packing ${\cal P}^*$ for $w'$ is unique. By the unique optimal solution case above, we can find ${\cal P}^*$ in $O(f(n)^{|A|+|B|})$ time.
We finally verify that ${\cal P}^*$ is actually shortest for the original uniform weight $w$.
Indeed, pick an arbitrary packing ${\cal P}$ not equal to ${\cal P}^*$. Then we have
\begin{eqnarray*}
1 &\leq& w'({\cal P})-w'({\cal P}^*)
\leq (2mn+2m-1)w({\cal P})-2mnw({\cal P}^*)\\
&\leq& 2mn(w({\cal P})-w({\cal P}^*))+(2m-1)w({\cal P}). 
\end{eqnarray*}
Hence we have
$$w({\cal P}) - w({\cal P}^*) \geq \frac{1}{2mn} - \frac{(2m-1)w({\cal P})}{2mn} \geq -1+\frac{1+w({\cal P})}{2mn}> -1,$$
where the second inequality follows from $w({\cal P}) \leq n$. 
Since both $w({\cal P})$ and $w({\cal P}^*)$ are integers, we have
$w({\cal P}) - w({\cal P}^*) \geq 0$. This means that ${\cal P}^*$ is shortest for $w$.
\end{proof}

\section{Related Results}
\subsection{NP-Completeness}
Here we verify that the perfect $(A+B)$-path packing problem,
the problem of deciding
the existence of a perfect $(A+B)$-path packing (with $|A|+|B|$ unfixed),
is intractable.

\begin{theo}\label{theo:npcomp}
The perfect $(A+B)$-path packing problem is NP-complete, even if $|B|=2$.
\end{theo}

\begin{proof}
Hirai and Pap \cite{hirai} proved that the following edge-disjoint paths
problem is NP-complete:
($*$) Given an undirected graph $G=(V,E)$ and $S,T \subseteq V$ with
$S\cap T =  \emptyset$ and $|S| = |T| = k$ and $a,b \in V\setminus (S \cup T)$,
find an edge-disjoint set ${\cal P}$ of paths $P_0,P_1,\dots,P_k$ such that
$P_0$ connects $a$ and $b$ and $P_i$ connects $S$ and $T$ $(i=1,2,\dots,k)$.
They gave a reduction from 3-SAT to the problem ($*$).
In their reduction \cite[Section 5.2.3]{hirai}, a solution is necessarily vertex-disjoint.
Moreover, 
one can see from the reduction that a set ${\cal P}$ of paths is a solution of~($*$) if and only if
${\cal P}$ is a perfect ($S\cup T + \{a,b\}$)-path packing.
Consequently the perfect $(A+B)$-path packing problem is also
NP-complete, even if $|B| =2$.
\end{proof}

\subsection{Other Path Packing via Hafnian}
In this subsection, we generalize our technique for solving other path packing problems and discuss its limitation. 
Let $G=(V,E)$ be a simple undirected graph. 
Let $T$ be a terminal set with even cardinality $|T| = 2\tau$.
As in Section \ref{sec:algo}, we assume that there is no edge joining $T$. 

To specify path packing problems, we introduce a notion of {\em perfect matching with parity (PMP)} on $T$, which is defined as
a set of pairs $(s_it_i,\sigma_i)$ $(i=1,\dots,\tau)$ such that $\bigcup_{i}\{s_i,t_i\}=T$ and $\sigma_i\in\{\mathrm{odd}, \mathrm{even}\}$ is a parity.
A perfect $T$-path packing ${\cal P}$ (a disjoint set of $\tau$ $T$-paths)
induces PMP $M_{\cal P}$: 
\[
M_{\cal P} := \{  (st, \sigma) \mid \mbox{${\cal P}$ has an $(s,t)$-path with
its length having the parity $\sigma$} \}.
\]
For a set ${\cal M}$ of PMPs,  
a {\em perfect ${\cal M}$-path packing} is 
a perfect $T$-path packing with $M_{\cal P} \in {\cal M}$.
We introduce {\em the shortest perfect ${\cal M}$-path packing problem} as the problem of finding a perfect ${\cal M}$-path packing of minimum size. 
Notice that an $(A+B)$-path packing corresponds to 
${\cal M}_{A+B}:=\{ M \cup M' \mid  \mbox{$M$:PMP on $A$, $M'$:PMP on $B$} \}$.

Next we consider a generalization of matrix $S'$.
As in Section \ref{sec:algo}, consider graph $H$, edge sets $E_t$ and $E'_t$, and matrix $S$ (with $A\cup B=T$). 
Suppose that $T=\{1,2,3,\dots, 2\tau\}$. 
For $p=(p_1,\dots,p_{2\tau}),q=(q_1,\dots,q_{2\tau})\in \mathbb{Z}^{2\tau}$, we define the matrix $S[p,q]$ from $S$ by 
\begin{equation*}
(S[p,q])_{ij}:=
\begin{cases}p_{t}s_{ij}&\mathrm{if}\: ij\in E_{t}\:\mathrm{for}\:t\in T, \\q_{t}s_{ij}&\mathrm{if}\:ij\in E'_{t}\:\mathrm{for}\:t\in T, \\s_{ij}&\mathrm{otherwise}. \end{cases}
\end{equation*}
For distinct $s,t\in T$ and parity $\sigma$, define $[p,q]_{st,\sigma}$ by
\begin{equation*}
[p,q]_{st,\sigma}:=
\begin{cases}
p_{s}p_{t}+q_{s}q_{t}&\mbox{if $\sigma=$ odd},\\
p_{s}q_{t}+q_{s}p_{t}&\mbox{if $\sigma=$ even}. 
\end{cases}
\end{equation*}
A set ${\cal M}$ of PMPs is said to be {\em h-representable} if there exist $N,k\in \mathbb{Z}_{>0}$, $n_i\in \mathbb{Z}_{\geq 0}$, $p^i,q^i\in \mathbb{Z}^{2\tau}$ for $i=1,\dots,N$ such that 
a PMP $M$ belongs to ${\cal M}$ if and only if $$\sum_{i=1}^Nn_i\prod_{(st,\sigma)\in M}[p^i,q^i]_{st,\sigma}\not\equiv 0\:\:\mathrm{mod}\:2^k.$$
In particular, the argument in Section \ref{sec:algo} says that ${\cal M}_{A+B}$ is h-representable with $N=1$, $k=\tau+1$, $n_1=1$, $p^1=(1,1,\dots,1)$ and $q^1=(1,\dots,1,-1,\dots,-1)$.
That is, $q^1$ has $1$ for the first $|A|$ entries and $-1$
 the remaining $|B|$ entries.
A generalization of Theorem \ref{theo:main} is the following. 

\begin{theo}\label{theo:hrep}
Suppose that a set ${\cal M}$ of PMPs is h-representable with parameters $N,k,n_i$, $p^i,q^i (i=1,2,\dots,N)$.
Then the shortest perfect ${\cal M}$-path packing problem can be solved in randomized polynomial time, provided $N$ and $k$ are fixed. 
\end{theo}
\begin{proof}
As in the proof of Lemma \ref{lem:hafs}, one can show
$$\displaystyle\sum_{i=1}^Nn_i\haf S[p^i,q^i]=\sum_{{\cal P}}\left[\sum_{i=1}^Nn_i\prod_{(st,\sigma)\in M_{\cal P}}[p^i,q^i]_{st,\sigma}\right]x^{w({\cal P})}(1+xf_{\cal P}(x)),$$
where ${\cal P}$ ranges over all perfect $T$-path packings. 
Therefore, if $G$ has a unique shortest perfect ${\cal M}$-path packing ${\cal P}^*$, 
then we can obtain ${\cal P}^*$ by computing $\sum_{i=1}^Nn_i\haf S[p^i,q^i]$ modulo $2^k$. 
This can be done in polynomial time provided $N$ and $k$ are fixed. 
As in Section \ref{sec:algo}, we obtain the randomized polynomial time algorithm for the general case. 
\end{proof}
We do not know a characterization of h-representable sets of PMPs.
We here discuss three interesting special cases, where odd and even are simply denoted by o and e respectively. 

\paragraph{Shortest two disjoint paths via hafnian modulo 4.}
First we return to the shortest two disjoint paths problem, which corresponds to $T=\{1,2,3,4\}$ and 
\[{\cal M}_2:=\{\{(12,\sigma_1),(34,\sigma_2)\}\mid \sigma_1,\sigma_2\in\{\mathrm{o}, \mathrm{e}\}\}.\]
We have seen that ${\cal M}_2$ is h-representable with $N=1=n_1=1$, $p^1=(1,1,1,1)$, $q^1=(1,1,-1,-1)$, and $k=3$. 
We present another economical h-representation. 
\begin{prop}
${\cal M}_2$ is h-representable with $N=1$, $k=2$, $n_1=1$, $p^1=(1,1,1,1)$, and $q^1=(0,1,-1,-1)$. 
\end{prop}
\begin{proof}
A direct calculation (e.g.,$[p^1,q^1]_{12,\mathrm{e}}[p^1,q^1]_{34,\mathrm{o}}=(1\cdot 1+0\cdot 1)\{1\cdot 1+(-1)\cdot (-1)\}=2$) shows
\begin{equation*}
\prod_{(st,\sigma)\in M}[p^1,q^1]_{st,\sigma}=
\begin{cases}
2&\mbox{if $M=\{(12,\mathrm{o}),(34,\mathrm{o})\},\{(12,\mathrm{e}),(34,\mathrm{o})\}$},\\
-2&\mbox{if $M=\{(12,\mathrm{o}),(34,\mathrm{e})\},\{(12,\mathrm{e}),(34,\mathrm{e})\}$},\\
0&\mbox{otherwise}.
\end{cases}
\end{equation*}
\end{proof}
In particular, modulo $4$ computation is sufficient. 
It might be interesting to compare with the original approach by Bj\"{o}rklund--Husfeldt~\cite{twopaths}:
their algorithm requires to compute permanents of three $n\times n$ matrices modulo $4$, whereas our algorithm with these parameters requires to compute the hafnian of one $2n\times 2n$ matrix modulo $4$.

\paragraph{Shortest odd two disjoint paths via four hafnians modulo 4.}
The hafnian approach can solve the shortest two disjoint paths problem with a parity constraint that the sum of the lengths of paths is odd. 
This problem corresponds to $T=\{1,2,3,4\}$ and ${\cal M}_{2,\mathrm{odd}}:=\{\{(12,\mathrm{o}),(34,\mathrm{e})\},\{(12,\mathrm{e}),(34,\mathrm{o})\}\}$. 

\begin{theo}
${\cal M}_{2,\mathrm{odd}}$ is h-representable with $N=4$, $k=2$,  $(n_1,n_2,n_3,n_4)=(1,1,-1,-1)$, and
\begin{eqnarray*}
p^1=(1,1,1,0),\quad q^1=(0,0,0,1),\\
p^2=(1,1,0,1),\quad q^2=(0,0,1,0),\\
p^3=(1,0,1,1),\quad q^3=(0,1,0,0),\\
p^4=(0,1,1,1),\quad q^4=(1,0,0,0).
\end{eqnarray*}
\end{theo}
\begin{proof}
\begin{table}[t]
\caption{Values of $C_i$. }\label{tab:partwo}
\begin{tabular}{|c|ccccc|}\hline
PMP&$C_1$&$C_2$&$C_3$&$C_4$&$C_1+C_2-C_3-C_4$\\\hline
$\{(12,\mathrm{o}),(34,\mathrm{o})\}$&0&0&0&0&0\\
$\{(12,\mathrm{o}),(34,\mathrm{e})\}$&1&1&0&0&2\\
$\{(12,\mathrm{e}),(34,\mathrm{o})\}$&0&0&1&1&-2\\
$\{(12,\mathrm{e}),(34,\mathrm{e})\}$&0&0&0&0&0\\
$\{(13,\mathrm{o}),(24,\mathrm{o})\}$&0&0&0&0&0\\
$\{(13,\mathrm{o}),(24,\mathrm{e})\}$&1&0&1&0&0\\
$\{(13,\mathrm{e}),(24,\mathrm{o})\}$&0&1&0&1&0\\
$\{(13,\mathrm{e}),(24,\mathrm{e})\}$&0&0&0&0&0\\
$\{(14,\mathrm{o}),(23,\mathrm{o})\}$&0&0&0&0&0\\
$\{(14,\mathrm{o}),(23,\mathrm{e})\}$&0&1&1&0&0\\
$\{(14,\mathrm{e}),(23,\mathrm{o})\}$&1&0&0&1&0\\
$\{(14,\mathrm{e}),(23,\mathrm{e})\}$&0&0&0&0&0\\\hline
\end{tabular}
\centering
\end{table}
One can verify the theorem from the value of $C_i:=\prod_{(st,\sigma)\in M}[p^i,q^i]_{st,\sigma}$ 
for $i=1,2,3,4$ and all PMPs $M$ on $T$, which are shown in Table \ref{tab:partwo}. 
\end{proof}

\paragraph{Non h-representability of 3-disjoint paths.}
A deep result by Robertson--Seymour~\cite{robertson} is that the $k$-disjoint paths problem is solvable in polynomial time (for fixed $k$) . 
One may naturally ask whether the shortest $k$-disjoint paths problem for $k\geq 3$ is solvable by this approach. 
Unfortunately our approach cannot reach the shortest 3-disjoint paths problem, which corresponds to $T=\{1,2,3,4,5,6\}$ and 
\[{\cal M}_3:=\{\{(12,\sigma_1),(34,\sigma_2),(56,\sigma_3)\}\mid \sigma_1,\sigma_2,\sigma_3\in \{\mathrm{o}, \mathrm{e}\}\}. \] 
\begin{theo}\label{theo:3imp}
${\cal M}_3$ is not h-representable. 
\end{theo}

We start with a preliminary argument. Let $\bi{1}:=(1,1,\dots, 1)$. 
For $\chi\in\{0,1\}^{2\tau}$, let $S(\chi):=S[\chi,\bi{1}-\chi]$.
Then $\haf S[p,q]$ can be expressed as a linear combination of $\haf S(\chi)$ over $\chi\in\{0,1\}^{2\tau}$:

\begin{lem}\label{hafbunkai}
$\haf S[p,q]=\displaystyle\sum_{\chi\in\{0,1\}^{2\tau}}\prod_{i=1}^{2\tau}\left\{\chi_ip_i+(1-\chi_i)q_i\right\}\:\haf S(\chi)$. 
\end{lem}
\begin{proof}
Each perfect matching of $H$ determines $\chi\in\{0,1\}^{2\tau}$ as: $\chi_i=1$ if and only if
node $i$ is matched to a node in $U$. Here $\chi$ is called the {\em type} of $M$. 
We classify all perfect matchings in terms of their types. 
One can verify
\[\displaystyle\sum_{M:\mathrm{type}\:\chi}\prod_{ij\in M} (S[p,q])_{ij}=\left[\prod_{i=1}^{2\tau}\left\{\chi_ip_i+(1-\chi_i)q_i\right\}\right]\haf S(\chi).\]
Thus we have the desired formula. 
\end{proof}
From Lemma \ref{hafbunkai}, in the definition of h-representability, it suffices to consider the case where $p=\chi$ and $q=\bi{1}-\chi$ for $\chi\in\{0,1\}^{2\tau}$. 
In this case, $\prod_{(st,\sigma)\in M}[p,q]_{st,\sigma}$ is $0$ or $1$. 
Let $[\chi]_{st,\sigma}:=[\chi,\bi{1}-\chi]_{st,\sigma}$. 
%Observe that $[\chi]_{st,\mathrm{o}}=1$ if $\chi_s=\chi_t$, and $[\chi]_{st,\mathrm{o}}=0$ otherwise, and $[\chi]_{st,\mathrm{e}}=0$ if $\chi_s=\chi_t$, and $[\chi]_{st,\mathrm{e}}=1$ otherwise.

\begin{proof}[Proof of Theorem $\ref{theo:3imp}$]
First consider the following six PMPs:
\begin{eqnarray*}
M_1:=\{(12,\mathrm{o}),(34,\mathrm{o}),(56,\mathrm{e})\},\quad
M_2:=\{(12,\mathrm{o}),(36,\mathrm{o}),(45,\mathrm{e})\},\\
M_3:=\{(14,\mathrm{o}),(23,\mathrm{o}),(56,\mathrm{e})\},\quad
M_4:=\{(14,\mathrm{o}),(36,\mathrm{o}),(25,\mathrm{e})\},\\
M_5:=\{(16,\mathrm{o}),(23,\mathrm{e}),(45,\mathrm{o})\},\quad
M_6:=\{(16,\mathrm{o}),(34,\mathrm{e}),(25,\mathrm{o})\}.
\end{eqnarray*}
Observe that $M_1$ is in ${\cal M}_3$ and other five PMPs are not in ${\cal M}_3$. For PMP $M$ and $\chi\in\{0,1\}^{6}$, define $b_{M,\chi}$ by
\[b_{M,\chi}:=\prod_{(st,\sigma)\in M}[\chi]_{st,\sigma}. \]
By computer calculation, we have verified the following 64 equations to hold;
\begin{equation}\label{mujun1}
b_{M_1,\chi}=b_{M_2,\chi}+b_{M_3,\chi}-b_{M_4,\chi}+b_{M_5,\chi}-b_{M_6,\chi}\quad (\chi\in\{0,1\}^6). 
\end{equation}

Next suppose that ${\cal M}_3$ is h-representable. 
Thanks to Lemma \ref{hafbunkai}, there exist $k\in\mathbb{Z}_{> 0}$ and $n_{\chi}\in\mathbb{Z}$ for $\chi\in\{0,1\}^6$
such that a PMP $M$ belongs to ${\cal M}$ if and only if
\begin{equation*}
\sum_{\chi\in\{0,1\}^6}n_{\chi}\prod_{(st,\sigma)\in M}[\chi]_{st,\sigma}\not\equiv 0\:\:\mathrm{mod}\:2^k.
\end{equation*}
In particular, it holds
\begin{equation*}\sum_{\chi\in\{0,1\}^6}n_{\chi}b_{M_j,\chi}\equiv 0\mod 2^k\:\:\:(j=2,3,4,5,6). \end{equation*}
By (\ref{mujun1}), we have
\begin{equation*}\sum_{\chi\in\{0,1\}^6}n_{\chi}b_{M_1,\chi}\equiv 0\mod 2^k. \end{equation*}
However this is a contradiction to $M_1\in{\cal M}_3$. 
\end{proof}

\section*{Acknowledgments}
We thank the referees for helpful comments.
The work was partially supported by JSPS KAKENHI Grant Numbers
25280004, 26330023, 26280004, 17K00029.
% 25280004 (Fujishige kaken 2013--2018)
% 26330023 (Hirai kaken 2014--2017)
% 17K00029 (Hirai Kaken 2017--2019)
% 26280004 (Murota kaken )

\appendix


\begin{thebibliography}{}

\bibitem{bjorklundSODA12}	
A. Bj\"{o}rklund: 
Counting perfect matchings as fast as Ryser,  
{\it Proceedings of the Twenty-Third Annual ACM-SIAM Symposium on Discrete Algorithms}, ACM, New York, (2012),  914--921.
 
\bibitem{twopaths}
A. Bj\"{o}rklund and T. Husfeldt: Shortest two disjoint paths in polynomial time, 
{\it Proceedings of 41st International Colloquium on Automata, Languages, and Programming}, Lecture Notes in Computer Science 8572, Springer-Verlag, Berlin, (2014), 211--222. 

\bibitem{schrijver2}
\'E. Colin de Verdi\`ere and A. Schrijver: Shortest vertex-disjoint two-face paths in planar graphs.
{\it ACM Transactions on Algorithms} 7 (2011), No. 19, 12 pp. 
\bibitem{gallai}
T. Gallai: Maximum-minimum S$\mathrm{\ddot{a}}$tze und verallgemeinerte Faktoren von Graphen, 
{\it Acta Mathematica Academiae Scientiarum Hungaricae} 12, (1961), 131--173.
\bibitem{hirai}
H. Hirai and G. Pap: Tree metrics and edge-disjoint S-paths, 
{\it Mathematical Programming} 147, (2014), 81--123. 
\bibitem{karzanov}
A. Karzanov: Edge-disjoint T-paths of minimum total cost, {\it Technical Report}, STAN-CS-92-1465, Department of Computer Science, Stanford University, Stanford, California, 1993, Available at http://alexander-karzanov.net/.
\bibitem{kobayashi2}
Y. Kobayashi and C. Sommer: On shortest disjoint paths in planar graphs. 
{\it Discrete Optimization} 7, (2010), 235--245. 
\bibitem{kobayashi}
Y. Kobayashi and S. Toyooka: Finding a shortest non-zero path in group-labeled graphs, 
{\it Algorithmica} 77, (2017), 1128--1142.

\bibitem{vazirani}
K. Mulmuley, U. V. Vazirani and V. V. Vazirani: Matching is as easy as matrix inversion, 
{\it Combinatorica} 7, (1987), 105--113.
\bibitem{robertson}
N. Robertson and P. D. Seymour: Graph minors. XIII. The disjoint paths problem, 
{\it Journal of Combinatorial Theory, Series} B 63, (1995), 65--110. 

\bibitem{schrijver}
A. Schrijver: {\it Combinatorial Optimization. Polyhedra and Efficiency}, Springer-Verlag, Berlin, 2003. 
\bibitem{seymour}
P. D. Seymour: Disjoint paths in graphs,
{\it Discrete Mathematics} 29, (1980), 293--309. 
\bibitem{shiloach}
Y. Shiloach: A polynomial solution to the undirected two paths problem, 
{\it Journal of the ACM} 27, (1980), 445--456. 
\bibitem{thomassen}
C. Thomassen: 2-linked graphs, 
{\it European Journal of Combinatorics} 1, (1980), 371--378.
\bibitem{valiant}
L. G. Valiant: The complexity of computing the permanent, 
{\it Theoretical computer science}, 8, (1979), 189--201. 
\bibitem{yamaguchi}
Y. Yamaguchi: Shortest disjoint non-zero A-paths via weighted matroid matching,
{\it Proceedings of the 27th International Symposium on Algorithms and Computation}, (2016), No. 63, 13 pp.
%{\it Mathematical Engineering Technical Reports}, METR 2015-20, University of Tokyo, 2015.

\end{thebibliography}
\end{document}